\documentclass[11pt]{amsart}

\usepackage{amssymb}
\usepackage{latexsym}
\usepackage{amsmath}
\usepackage{amsthm}
\usepackage{amsfonts}
\usepackage{color}
\usepackage[hidelinks]{hyperref}
\usepackage{pdfsync}
\usepackage[usenames,dvipsnames]{pstricks}
\usepackage{epsfig}
\usepackage{pst-grad} 
\usepackage{pst-plot} 
\usepackage{accents}

\usepackage{dsfont}

\newcommand{\R}{\mathbb{R}}

\newcommand{\I}{\mathcal{I}}

\newcommand{\B}{\mathcal{B}}

\newcommand{\M}{\mathcal{M}}

\newcommand{\Lambd}{\Gamma}
\newcommand{\lambd}{\gamma}

\theoremstyle{plain}
\newtheorem{defi}{Definition}[section]
\newtheorem{prop}[defi]{Proposition}
\newtheorem{teo}[defi]{Theorem}
\newtheorem{cor}[defi]{Corollary}
\newtheorem{lema}[defi]{Lemma}

\newtheorem{remark}[defi]{Remark}

\theoremstyle{definition}

\theoremstyle{remark}

\numberwithin{equation}{section}

\begin{document}

\title[]{Large harmonic functions for fully nonlinear fractional operators}

\author[]{Gonzalo D\'avila}
\address{
Gonzalo D\'avila: Departamento de Matem\'atica, Universidad T\'ecnica Federico Santa Mar\'ia \\
Casilla: v-110, Avda. Espa\~na 1680, Valpara\'iso, Chile
}
\email{gonzalo.davila@usm.cl}

\author[]{Alexander Quaas}
\address{
Alexander Quaas: Departamento de Matem\'atica, Universidad T\'ecnica Federico Santa Mar\'ia \\
Casilla: v-110, Avda. Espa\~na 1680, Valpara\'iso, Chile}
\email{alexander.quaas@usm.cl}

\author[]{Erwin Topp$^1$}
\address{
Erwin Topp: Departamento de Matem\'atica y C.C., Universidad de Santiago de Chile,
Casilla 307, Santiago, Chile.
}
\email{erwin.topp@usach.cl}



\date{\today}
\footnote{ \ Corresponding author}

\begin{abstract}
We study existence, uniqueness and boundary blow-up profile for fractional harmonic functions on a bounded smooth domain $\Omega \subset \R^N$. We deal with harmonic functions associated to uniformly elliptic, fully nonlinear nonlocal operators, including the linear case
$$
(-\Delta)^s u = 0 \quad \mbox{in} \ \Omega,
$$ 
where $(-\Delta)^s$ denotes the fractional Laplacian of order $2s \in (0,2)$. We use the viscosity solution's theory and Perron's method to construct harmonic functions with zero exterior condition in $\bar \Omega^c$, and boundary blow-up profile 
$$
\lim_{x\to x_0, x \in \Omega}\mathrm{dist}(x, \partial \Omega)^{1-s}u(x)=h(x_0), \quad \mbox{for all} \quad x_0\in \partial \Omega,
$$
for any given boundary data $h \in C(\partial \Omega)$. Our method allows us to provide blow-up rate for the solution and its gradient estimates. Results are new even in the linear case.
\end{abstract}

\maketitle

\section{Introduction.}

In this paper, we consider $\Omega \subset \R^N$ a bounded domain with $C^2$ boundary, and deal with the existence, uniqueness and qualitative properties of fractional harmonic  functions on $\Omega$ that blow-up near the boundary $\partial \Omega$. 

For $s \in (0,1)$ fixed, we denote $\Delta^s$ the fractional Laplacian of order $2s$, explicitly given by
\begin{equation}\label{fracLapl}
\Delta^s u(x) = -(-\Delta)^su(x) = {C_{N, s} \mathrm{P.V.} \int_{\R^N} \frac{u(x) - u(z)}{|x - z|^{N + 2s}}dz},
\end{equation}
for measurable $u: \R^N \to \R$ and $x \in \R^N$ for which the integral makes sense. Here $P.V.$ stands for the Cauchy Principal Value and $C_{N,s} > 0$ is a well known normalizing constant, see~\cite{Hitch}. In particular, for sufficiently regular and summable function $u: \R^N \to \R$, we have $(-\Delta)^s u \to -\Delta u$ locally uniform in $\R^N$ as $s \to 1$.

We say that a function $u$ is $s$-harmonic in $\Omega$ if $\Delta^s u(x) = 0$ for all $x \in \Omega$. The notion of $s$-harmonicity is a subject of study in various mathematical settings, such  probability \cite{bass2}, harmonic and potential analysis~\cite{Bbook, land}, and they naturally arise in a broad range of applications, see for example 
for a review of some applications in \cite{Hitch}. 

In the context of the analysis of PDE, there are different notions of solution for the equation defining $s$-harmonicity, namely classical solutions, weak solutions, viscosity solutions, and others. Here we follow the viscosity solution's approach of Barles, Chasseigne and Imbert~\cite{BChI}, and Caffarelli and Silvestre~\cite{CS}.

In parallel to the classical second-order setting, $s$-harmonic functions play a central role in the study of more general elliptic equations due to its rich set of properties, such as mean value and representation formulas, positivity, regularity, (strong) maximum principles, Harnack inequality, and many others analytic and geometric tools. Nevertheless, there are some evident differences with the second-order case. A structural one as to do with the definition of the fractional Laplacian, from which $s$-harmonic functions need to be specified in $\Omega^c$. This feature implies some unexpected phenomena: for instance, it is proven in~\cite{DSV} the density of $s$-harmonic functions into $C_{loc}^k$, a phenomenon that does not appear in the second-order setting.

Another intriguing aspect of $s$-harmonicity has to do with the existence of large $s$-harmonic functions. This fact is studied in Abatangelo~\cite{aba}, where the author proves the existence and uniqueness of weak solutions to the Dirichlet problem
\begin{align}\label{eqzero}
\left\{\begin{array}{rll}
\Delta^s u & = f & \mbox{in} \  \Omega,\\
u & = g &  \mbox{in} \ \bar \Omega^c, \\
\lim \limits_{z \to x, \ z \in \Omega} \mathbb M(z) u(z) & = h(x) & \mbox{for} \ x \in \partial \Omega,
\end{array}\right.
\end{align}
for some appropriate data $f,g,h$. In fact, this data can be of measure type in~\cite{aba}, but for simplicity we may think on continuous and bounded functions, that is $f \in C_b(\Omega), g \in C_b(\bar \Omega^c)$ and $h \in C(\partial \Omega)$.

Here we concentrate in the case $f \equiv 0$. Notice that we have split the usual exterior condition on $\Omega^c$ in two parts: the \textsl{(actual) exterior} data $g$ in $\bar \Omega^c$, and the \textsl{(limit) boundary condition} $h$ in $\partial \Omega$. The function $\mathbb M$ depends on $\Omega$ and $s$, and it is referred in~\cite{aba} (up to an explicit modification) as \textsl{the Martin kernel} of the fractional Laplacian in $\Omega$. For example, when $\Omega$ is the unit ball $B_1$, this function has the explicit structure 
$$
\mathbb M(x) = \tilde C_{N,s} (1 - |x|)^{1 - s}, \quad x \in B_1,
$$
for some suitable constant $\tilde C_{N,s} > 0$, see Lemma 3.1.5 in~\cite{aba}. Notice that for $x \in B_1$ we have $\mathrm{dist}(x, \partial B_1) = 1 - |x|$, from which the solution $u$ to~\eqref{eqzero} may blow-up near the boundary. Moreover, it tends to $-\infty$ near boundary points where $h$ is negative. 
The results in~\cite{aba} are based on a representation formula for the solutions through Green's function associated to $\Delta^s$ and integration by parts techniques. As we mentioned above, this formulation has the advantage to deal with very irregular datum $f,g,h$ in~\eqref{eqzero}, but has a restricted applicability to more general operators.

\medskip

It is our aim in this paper to deal with fully nonlinear extensions of problem~\eqref{eqzero}, and for this we introduce our main assumptions.
We consider ellipticity constants $0 < \lambd \leq \Lambd < +\infty$ and the class $\mathcal L_0$ of measurable even kernels $K : \R^N \setminus \{ 0\} \to \R$ such that $K(y) = K(-y)$ for all $y$ and that satisfy the ellipticity condition
\begin{equation}\label{ellipticity}
\gamma |y|^{-(n + 2s)} \leq K(y) \leq \Gamma |y|^{-(N + 2s)}, \quad y \neq 0,
\end{equation}

For each $K \in \mathcal L_0$, we consider the linear operator
\begin{equation}\label{lineal}
L_K u(x) := \mathrm{P.V.} \int_{\R^n} [u(x + y)-u(x)] K(y) dy, 	
\end{equation}
which is well defined for measurable $u: \R^N \to \R$ satisfying adequate regularity assumptions on $x$ and weighted integrability condition at infinity; typically $u \in C^{1,1}$ in a neighborhood of $x$ and $u \in L^1_\omega(\R^N)$, where for measurable set $E \subseteq \R^N$ we denote
$$
L^1_\omega(E) := \left\{ u \in L^1_{loc}(\R^N) : \int_{E} |u| \omega < +\infty \right\}, \ \mbox{with} \ \omega(y) := \frac{1}{(1 + |y|)^{N + 2s}}.
 $$
 
 We focus on the subclass $\mathcal K \subset \mathcal L_0$ of kernels with the form
\begin{equation}\label{Kestable}
K(z) = \frac{a(\hat z)}{|z|^{N + 2s}}, \quad z \neq 0, \ \hat z = z/|z|,
\end{equation}
for some nonnegative, measurable function $a: S^{N - 1} \to \R$. In this setting, condition~\eqref{ellipticity} turns out to be $\gamma \leq a \leq \Gamma$. 

Let us consider a  two-parameter family of kernels $\{ K_{ij}\}_{i \in I, j \in J} \subseteq \mathcal K$, and denote $L_{ij} := L_{K_{ij}}$ with $K_{ij}(z) = a_{ij}(\hat z) |z|^{-(N + 2s)}$ for $z \neq 0$. We can now define nonlinear, nonlocal operators of the Bellman-Isaacs form 
\begin{align} \label{opp}
	 \I u(x):= \inf\limits_{i\in I} \sup\limits_{j \in J} L_{ij}u (x).
\end{align}

In terms of applications, it is known that linear operators  $L_K$ as defined here act as infinitesimal generator of $2s$-stable L\'evy processes, see~\cite{Duke}. Nonlinear operators like~\eqref{opp} arise naturally in the context of stochastic control problems and games, see~\cite{Pham, OS}.

Here and in the rest of the paper $d(x) = \mathrm{dist}(x, \partial \Omega)$ will denote the distance to the boundary of  $\Omega$. We can now state our first  main result.
\begin{teo}\label{teo1}
Let $\Omega \subset \R^N$ be a bounded domain with $C^2$ boundary. Let $\I$ be an operator with the form~\eqref{opp}. Then, for every $h \in C(\partial \Omega)$, there exists  a unique viscosity solution $u \in C^{1, \alpha}(\Omega)$ of the Dirichlet problem
\begin{align}\label{eq}
\left\{\begin{array}{rll}
\I u & = 0 & \mbox{in} \  \Omega,\\
u & = 0 &  \mbox{in} \  \bar \Omega^c, \\
\lim \limits_{z \to x, \ z \in \Omega} d(z)^{1 - s} u(z) & = h(x) & \mbox{for} \ x \in \partial \Omega. 
\end{array}\right.
\end{align}

Moreover, there exists a modulus of continuity $\tilde m$ depending on the data such that, for each $x_0 \in \partial \Omega$ and all $x \in \Omega$ close to $x_0$, we have
$$
|d^{1 - s}(x)u(x) - h(x_0)| \leq \tilde m(|x - x_0|).
$$
\end{teo}

Here we use the viscosity solution's theory to deal with the problem, see section~\ref{secbasic} for definitions and preliminaries. The existence follows by Perron's method together with comparison principles, which leads to the result after the construction of appropriate barriers that allows us to get the boundary condition. Those barriers are built up by the use of the function $x \mapsto d^{s-1}(x)$ for $x \in \Omega$, and equal to zero in $\Omega^c$, which is close to be $s$-harmonic, see~\cite{CFQ, DQTe}; together with the arguments of Capuzzo-Dolcetta, Leoni and Porretta~\cite{CDLP} (see also \cite{crandall}) to provide the blow-up boundary profile to the solution. Interior regularity is a consequence of well-known elliptic estimates, see Caffarelli and Silvestre~\cite{CS}.

We have chosen to  treat the (simpler) case of fractional harmonic functions with zero exterior condition. The non-homogeous case can be treated along the same lines presented here: for instance, in the case of~\eqref{eqzero}, it is possible to construct large solutions in combination with the bounded solution $v$ for the Dirichlet problem $\Delta^s v = f$ in $\Omega$, with exterior condition $v = g$ in $\Omega^c$ if $f$ and $g$ are, say, bounded and continuous, see~\cite{BChI}. The boundedness of $v$ does not disturb the limit boundary condition. In the nonlinear case, this procedure can be performed by using the extremal Pucci operators associated with the family $\mathcal K$, namely
$$
\mathcal M^+ u = \sup_{K \in \mathcal K}  L_K u; \quad \mathcal M^- u = \inf_{K \in \mathcal K}  L_K u,
$$
further details can be see in Corollary \ref{corfgh} and its proof.

The techniques introduced here can be applied to other family of kernels where the ``fundamental exponent" $s - 1$ changes into some $\tilde s - 1$ where $\tilde s \in (0,1)$ depends on $s$, the ellipticity constants and the family of kernels defining the nonlocal operator, see~\cite{FQ, Duke}.

%


\medskip

Our second main result states gradient estimates on the boundary.
\begin{teo}\label{teo2}
Assume $h \in C(\partial \Omega)$ and let $u$ be the unique solution to~\eqref{eq}. Then, there exists  constants $C_1 > 0$ and $\delta > 0$ such that
$$
|Du(x)| \leq C_1 d^{s-2}(x) \quad \mbox{for all} \ x \in \Omega_\delta,
$$
where $\Omega_\delta=\{x\in\Omega\text{ s.t. }\text{dist}(x,\partial\Omega<\delta)\}$. 
Furthermore, if  $\inf_{\partial \Omega} |h| > 0$ then there exists a constant $C_1$ such that
$$
C_2 d(x)^{s - 2} \leq |Du(x)| \leq C_1 d^{s-2}(x) \quad \mbox{for all} \ x \in \Omega_\delta.
$$
\end{teo}

In the previous theorem, upper and lower bound are obtained in different ways. For the upper bound, we rescale the solution near the boundary and employ $C^{1,\alpha}$ interior estimates of Kriventsov~\cite{K}. The lower bound is obtained by a contradiction argument together with scaling and compactness properties of the problem, leading us to a limit problem set up in the half-space that  is not compatible 
with the contradiction assumption.

This estimates can be employed in the extension of this result to semilinear problems involving nonlinearities like $f(u, D u, x)$, see~\cite{CFQ, DQTe}, but we do not pursue in this direction here. 

\medskip

The paper is organized as follows: in Section~\ref{secbasic} we introduce the basic notation and the notion of solution. In Section \ref{sec2} we discuss Perron's method and prove a comparison principle suitable for our purposes. In Section \ref{sec3} we present technical estimates which are going to be useful to construct the barriers. Section \ref{sec4} is completely devoted to the proof of Theorem \eqref{teo1}. Finally, in Section \ref{sec6} we provide the proof of Theorem \eqref{teo2}.


\section{Basic notation and notion of solution.}\label{secbasic}

From now on, for $\tau \in \R$, and $x \in \Omega$, we use the notation
$$
d^\tau(x) = \left \{ \begin{array}{ll} d(x)^\tau & \quad \mbox{if} \ x \in \Omega, \\ 0 & \quad \mbox{if} \ x \in \Omega^c. \end{array} \right .
$$

We start with some notation. For $r > 0$ and $x \in \R^N$ we denote $B_r(x)$ the ball with center $x$ and radius $r$.

For $\delta>0$ we will denote 
\[
\Omega_{\delta}=\{x\in\Omega : \text{dist}(x,\partial \Omega)<\delta\}.
\]

With some abuse of notation, we will denote the extremal operators in the class $\mathcal K$ by $\mathcal M^\pm=\mathcal M^\pm_{\mathcal K}$, see Caffarelli and Silvestre~\cite{CS}.

Now we present our notion of boundary blow-up viscosity solution. For this, we consider $\Omega \subset \R^N$, $f \in C(\Omega), h \in C(\partial \Omega)$ and $g \in C(\bar \Omega^c) \cap L_\omega(\R^N)$, and look for the problem
\begin{align}\label{eqdefsol}
\left\{\begin{array}{rll}
\I u & = f & \mbox{in} \  \Omega,\\
u & = g &  \mbox{in} \ \bar \Omega^c, \\
\lim \limits_{z \to x, \ z \in \Omega} d(z)^{1 - s} u(z) & = h(x) & \mbox{for} \ x \in \partial \Omega.
\end{array}\right.
\end{align}

We start with the definition of viscosity solution in the domain $\Omega \subseteq \R^N$.
\begin{defi}\label{defi1}
We say that $u \in L^1_{\omega}(\R^N) \cap L^\infty_{loc}(\Omega)$, upper semicontinuous (resp. lower semicontinuous) in $\Omega$, is a viscosity subsolution (resp. supersolution) to the equation
\begin{equation}\label{Iu=f}
\I u = f \quad \mbox{in} \ \Omega,
\end{equation}
if for each $x_0 \in \Omega$, $\delta > 0$ such that $B_\delta(x_0) \subset \Omega$, and each $\varphi \in C^2(B_\delta(x_0))$ such that $x_0$ is a maximum point (resp. minimum point) of $u - \varphi$ in $B_\delta(x_0)$, then
\begin{equation*}
\I_\delta(u, \varphi, x_0)  \geq (resp. \ \leq) f(x_0),
\end{equation*}
where we have denoted
$$
\I_\delta(u, \varphi, x_0) := \inf_{i \in I} \sup_{j \in J} \{ L_{ij}[B_\delta(x_0)] \varphi(x_0) + L_{ij}[B_\delta^c(x_0)]u(x_0) \},
$$
and 
for each $A \subseteq \R^N$ measurable, we have denoted
$$
L_{ij}[A]u(x) = \mathrm{P.V.} \int_{A} [u(y) - u(x)]K(x - y)dy.
$$ 

A viscosity solution $u \in C(\Omega) \cap L^1_\omega(\R^N)$ to~\eqref{Iu=f}  is a function which is a simultaneously a viscosity sub and supersolution to the problem.
\end{defi}

We notice that each classical solution $u \in C^2(\Omega) \cap L^1_\omega(\Omega)$ to~\eqref{Iu=f} is a viscosity solution to the problem. The requirement that $B_\delta(x_0) \subset \Omega$ is not restrictive, in the sense that if $\varphi \in C^2(\Omega \cap B_\delta(x_0)) \cap L^1_\omega(\R^N)$ is a test function for $u$ (say, subsolution) in $B_\delta(x_0)$ and the inequality
$$
\I_{\delta'}(u, \varphi, x_0) \geq f(x_0)
$$
holds for $\delta' < \delta$, then $\I_{\delta}(u, \varphi, x_0) \geq f(x_0)$.

We say that $u$ is a strict subsolution to the problem if the exterior and boundary conditions in the definition above are satisfied, and there exists $\epsilon > 0$ such that, for each $x_0 \in \Omega$ and each test function $\varphi$ as above, we have
\begin{equation*}
\I_\delta(u, \varphi, x_0) \geq \epsilon.
\end{equation*}

In the same way, we define strict supersolution. We say that $u$ is a local strict subsolution if for each $\Omega' \subset \subset \Omega$, $u$ is a strict subsolution in $\Omega'$. In the same way it is defined strict, and locally strict supersolution.

Our definition is a natural extension of the definition given by Barles, Chasseigne and Imbert in~\cite{BChI}.
\begin{defi}\label{defDirichlet}
We say that $u \in L^1_{\omega}(\R^N) \cap L^\infty_{loc}(\Omega)$, upper semicontinuous (resp. lower semicontinuous) in $\Omega$, is a viscosity subsolution (resp. supersolution) to the Dirichlet problem~\eqref{eqdefsol} if $u$ is a viscosity subsolution (resp. supersolution) to~\eqref{Iu=f} in the sense of Definition~\ref{defi1}, $u \leq (resp. \ \geq) \ g$ in $\bar \Omega^c$,  and for all $x \in \partial \Omega$ we have
$$
\limsup \limits_{z \to x, z \in \Omega} d^{1-s}(z) u(z) \leq h(x), \ (resp. \liminf \limits_{z \to x, z \in \Omega} d^{1-s}(z) u(z) \geq h(x)).
$$

A viscosity solution to~\eqref{eqdefsol} is a function which is a viscosity sub and supersolution to~\eqref{eqdefsol} simultaneously.
\end{defi}

\section{Perron's method and Comparison principles}\label{sec2}
We start this section we employ the following Perron type result.
\begin{prop}\label{Perron}
Let $\bar U, \underbar U \in L^1_\omega(\R^N) \cap C(\R^N \setminus \partial \Omega)$ be respectively super and subsolution to~\eqref{Iu=f} such that  $\underline U \leq \bar U$ in $\Omega$, $\bar U = \underbar U = g$ in $(\bar \Omega)^c$, and such that one of them is strict.
Then, there exists a function $u$ solving~\eqref{Iu=f}, with $u = g$ in $(\bar \Omega)^c$ and such that $\underline U \leq u \leq \bar U$ in $\Omega$.
\end{prop}

The proof  of this result can be found in \cite{DQTe} by stability of viscosity solutions for a sequence of problems defined in inner subdomains $\Omega_k = \Omega \setminus \Omega_{1/k}$ and take limit as $k \to \infty$. 

\begin{prop}\label{propcomparison}
Let $\Omega \subset \R^N$ be a bounded domain with smooth boundary. Let $u$ be a viscosity subsolution to the Dirichlet problem~\eqref{eqdefsol} with $f, g, h \equiv 0$.
Then $u\leq0$ in $\Omega$.
\end{prop}

\begin{proof}
We start by proving that for each $\kappa > 0$, there exists $\epsilon, \eta > 0$ and $\tau \in (0,s)$ such that, the function
$$
\tilde u = u - \epsilon d^{s - 1}_+ - \eta d_+^\tau - \kappa \chi_\Omega
$$
is a strict subsolution to the problem. Here, $\chi_\Omega$ denotes the indicator function of $\Omega$, and we have assumed that the distance function $d$ is extended as a positive, $C^2$ function in $\Omega$.

Let $x \in \Omega$ and $\varphi \in C^2$ touching $\tilde u$ from above at $x$ inside the ball $B_\delta(x)$. We can assume that $\delta < d(x)$. Then, we have that the function $\tilde \varphi = \epsilon d^{s - 1} + \eta d^\tau + \kappa$ is a test function for $u$, from which we use the viscosity inequality to conclude that
\begin{equation*}
\I_\delta(u,\varphi,x) + \epsilon \M^+ d^{s - 1}(x) + \eta \M^+ d^\tau (x) + \kappa \M^+ \chi_\Omega (x) \geq 0.
\end{equation*} 

Using the estimates from Proposition 3.1 in~\cite{DQTe}, we have the existence of a constant $C > 0$ (just depending on $N, \Omega, s$ and the ellipticity constants) such that
\begin{equation*}
\M^+ d^{s - 1}(x) \leq C d^{-1}(x), \quad \mbox{for all} \ x \in \Omega.
\end{equation*}

Similarly, using the estimates in Lemma 3.1 in~\cite{DQT}, there exists $\tau > 0$ small enough and $c(\tau) < 0$ such that
\begin{equation*}
\M^+ d^{\tau}(x) \leq c(\tau) d^{\tau - 2s}(x), \quad \mbox{for all} \ x \in \Omega_{\delta_0},
\end{equation*}
 meanwhile, for $x$ such that $d(x) \geq \delta_0$ we have $\M^+ d^{\tau}(x) \leq C$ for some $C > 0$. 
 
 Finally, a direct computation shows the existence of $c > 0$ such that
 \begin{equation*}
 \M^- \chi_\Omega(x) \leq -c \ \mathrm{diam}(\Omega)^{-2s} \quad \mbox{for all} \ x \in \Omega.
 \end{equation*}

Gathering the above estimates, we conclude that for all $d(x) \leq \delta_0$ we have
\begin{equation*}
\I_\delta (u, \varphi, x) \geq \kappa c \ \mathrm{diam}(\Omega)^{-2s} - C \epsilon d^{-1}(x) - \eta c(\tau) d^{\tau - 2s}(x),
\end{equation*}
from which, taking $\tau$ small such that $\tau - 2s < -1$ we conclude that
\begin{equation*}
\I_\delta (u, \varphi, x) \geq \kappa c \ \mathrm{diam}(\Omega)^{-2s} > 0 \quad \mbox{for all} \ x \in \Omega_{\delta_0}.
\end{equation*}

For $d(x) > \delta_0$, we have
\begin{equation*}
\I_\delta (u, \varphi, x) \geq \kappa c \  \mathrm{diam}(\Omega)^{-2s} - C_{\delta_0} (\epsilon + \eta),
\end{equation*}
from which, taking $\epsilon, \eta$ small enough, we conclude the result.

With this result at hand, we consider $u$ as in the statement of the proposition and assume by contradiction that 
\[
\sup\limits_{\Omega}u>0.
\]

Then, for $\kappa$ sufficiently small in terms of the above supremum, for $\epsilon, \eta, \tau$ as above, we have $\tilde u$ is a strict subsolution for the problem, $\tilde u \leq 0$ in $(\bar \Omega)^c$, and 
\[ 
0 < \sup_{\Omega} \tilde u = \max_{\Omega} \tilde u < +\infty,
\]

Thus, using the constant function equal to $M := \max_{\Omega} \tilde u = \tilde u(x_0)$ as test function for $\tilde u$, for each $\delta > 0$ small we have
\begin{equation*}
\I_\delta(u, M, x_0) \geq c \kappa > 0, 
\end{equation*}
but by its very definition, we have $\I_\delta (u, M, x_0) \leq 0$, from which we arrive at a contradiction.
%
%
%
%
\end{proof}

%
\section{Barriers}\label{sec3}

In what follows, for $a > 0$, we adopt the usual notation $O(a)$ to denote a quantity for which there exists a constant $C > 0$ not depending on $a$, for which
$$
-C a \leq O(a) \leq Ca.
$$

Notice, in particular, that for $a, b > 0$ we have $O(a) O(b) = O(ab)$.

\begin{lema}\label{lematec}
Let $\alpha \in (0,2s)$, $x_0 \in \partial \Omega$ and denote $\xi_\alpha(x) = |x - x_0|^\alpha$. Then, for each $\tau \in (-1, 2s)$, $K \in \mathcal K$, and each $x \in \Omega$ close to the boundary, we have
\begin{equation}\label{formulaprod}
L_K (d^{\tau} \xi_\alpha)(x) = \xi_\alpha(x) \rho^{\tau - 2s} \Big{(} c_K(\tau) + O(\rho^s) \Big{)} + O(\rho^{\tau + \alpha - 2s}) + O(\rho^\tau), 
\end{equation}
where we have denoted $\rho = d(x)$, and 
$$
c_K(\tau) = \mathrm{P.V.} \int_{\R^N} [(y_N)_+^\tau - 1] K(y - e_N)dy.
$$
\end{lema}

\begin{proof} For simplicity, we omit the dependence of $K$ on $L_K$ and of $\alpha$ in $\xi_\alpha$. 

Recalling $\rho = d(x)$, after a rotation and translation, we can assume $x = \rho e_N$, and that its projection to the boundary is the origin. For $\eta > 0$, we denote the cylinder $Q_\eta = B_\eta' \times (-\eta, \eta)$. We eventually fix $\eta$ small enough, but independent of $\rho$. By the compactness of $\partial \Omega$, we have the existence of $\eta \in (0,1)$ small enough (just depending on $\Omega$) such that $Q_\eta \cap \partial \Omega$ can be parametrized as $\{ (x', \psi(x')) : x' \in B_\eta' \}$ with $\psi \in C^2$, and such that $\psi(0') = 0$, $D\phi(0') = 0$. 

Using the product formula, for each $x \in \Omega$ close to the boundary, we have
$$
L (d^\tau \xi)(x)  = \xi(x) L d^\tau(x) + d^\tau(x) L \xi(x) + 2 \B(d^\tau, \xi)(x),
$$
where $\B = \B_K$ is given by
\begin{equation}\label{bilineal}
\B(f,g)(x) = \frac{1}{2} \int_{\R^N} (f(y) - f(x))(g(y) - g(x)) K(x - y)dy,
\end{equation}
for adequate $f, g: \R^N \to \R$.

Using Proposition 3.1 in~\cite{DQTe}, we have
\begin{align}\label{1}
Ld^\tau(x) = \rho^{\tau - 2s} (c_K(\tau) + O(\rho^{s})),
\end{align}
where the $O$-term depend only on $N, \Omega, s, \tau$ and the ellipticity constants. 

Using Lemma 3.3 in~\cite{DQT}, we have
\begin{align}\label{2}
L \xi(x) = O(1) |x - x_0|^{\alpha - 2s},
\end{align}
where the $O$-term depends on $N, \alpha, s$ and the ellipticity constants. Since $|x - x_0| \geq \rho$ and $\alpha < 2s$, we get the first two terms in the right-hand side of~\eqref{formulaprod}.

From now on, we concentrate in the bilinear form. For a set $A \subset \R^N$, we denote $\B[A]$ the operator $\B$ defined in~\eqref{bilineal}, but with domain of integration $A$. Using the boundedness of $\xi$ and the integrability of $d^\tau$, it is easy to see that
\begin{align}\label{bilineal2}
\B(d^\tau, \xi)(x) = & \B[Q_\eta](d^\tau, \xi)(x) + \B[Q_\eta^c](d^\tau, \xi)(x).
\end{align}

We assume $\rho \leq \eta/4$, from which $x \in Q_\eta$ and $\mathrm{dist}(x, Q_\eta^c) \geq \rho/2$.

For the integral over $Q_\eta^c$, we see that
\begin{align*}
\B[Q_\eta^c] (d^\tau, \xi)(x) = & \int_{Q_\eta^c \cap \Omega} d^\tau(y)\xi(y)K(x - y)dy -\xi(x)\int_{Q_\eta^c \cap \Omega} d^\tau(y) K(x - y)dy  \\
& -\rho^\tau \int_{Q_\eta^c} \xi(y) K(x - y)dy + \rho^\tau \xi(x) \int_{Q_\eta^c} K(x - y)dy, 
\end{align*}
from which, using that $\tau > -1$, $\alpha < 2s$ and that $\Omega$ is bounded, we arrive at
\begin{align}\label{BQetac}
\B[Q_\eta^c] (d^\tau, \xi)(x) =  O(1) + O(\rho^\tau),
\end{align}
where the $O$-term depends on the data and $\eta$.

Notice that if $\tau = 0$, then
\begin{equation*}
\B[Q_\eta](d^\tau, \xi)(x) = -\int_{\Omega^c \cap Q_\eta} (\xi(y) - \xi(x))K(x-  y)dy = O(\rho^{\alpha - 2s}),
\end{equation*}
which concludes the case $\tau = 0$. 

Now we deal with the first term in the right-hand side of~\eqref{bilineal2} for $\tau \neq 0$. For this, given $a > 0$ we denote $Q_a(x) = x + Q_a$, and we write 
$$
\tilde Q = Q_\eta \setminus Q_{\frac{\rho}{2\sqrt{N}}}(x)
$$ 
for simplicity. Thus, we split the integral as
\begin{align}\label{BQeta}
\B[Q_\eta](d^\tau, \xi)(x) = \B[\tilde Q](d^\tau, \xi)(x) + \B[Q_{\frac{\rho}{2\sqrt{N}}}(x)](d^\tau, \xi)(x)
\end{align}

If $y \in Q_{\frac{\rho}{2\sqrt{N}}}(x)$, using that $|x - x_0| \geq \rho$ we have
\begin{align*}
& |\xi(y) - \xi(x)| \leq C \rho^{\alpha - 1} |x - y|, \\
& |d^\tau(y) - d^\tau(x)| \leq C \rho^{\tau - 1} |x - y|,
\end{align*}
for some $C > 0$ just depending on the data. Thus
\begin{equation*}
\B[Q_{\frac{\rho}{2\sqrt{N}}}(x)](d^\tau, \xi)(x) = O(\rho^{\alpha + \tau - 2}) \int_{Q_{\frac{\rho}{2\sqrt{N}}}(x)} |x - y|^2 K(x - y) dy, 
\end{equation*}
from which we conclude that
\begin{equation}\label{BQint}
\B[Q_{\frac{\rho}{2\sqrt{N}}}(x)](d^\tau, \xi)(x) = O(\rho^{\alpha + \tau - 2s}),
\end{equation}
where the $O$-term just depends on the data.

Now, for the integral over $\tilde Q$, using the H\"older continuity of $\xi$ we have
\begin{align*}
\B[\tilde Q](d^\tau, \xi)(x) = & O(1) \int_{\tilde Q} \frac{d^{\tau}(y) - \rho^\tau}{|x - y|^{N + 2s - \alpha}}dy =: O(1) I,
\end{align*}
%
%
where the $O$-term just depends on the data. Observe that
\begin{equation}\label{lowerI}
I \geq -\rho^\tau \int_{\tilde Q} |x - y|^{-(N + 2s) + \alpha} dy \geq -C \rho^{\tau + \alpha - 2s},
\end{equation}
for some $C > 0$ just depending on the data and $\eta$.

For the upper bound, we see that
\begin{equation*}
I \leq \int_{\tilde Q \cap \Omega} \frac{d^\tau(y) - \rho^\tau}{|x - y|^{N + 2s - \alpha}}dy.
\end{equation*}

Now, using Lemma 3.1 in~\cite{CFQ}, we have the existence of a constant $C_\Omega > 0$ just depending on $\Omega$ such that, for each $y = (y', y_N) \in \Omega$ close to $\partial \Omega$, we have that $y_N > \psi(y')$ and 
\begin{equation*}
(y_N - \psi(y'))(1 - C_\Omega |y'|^2) \leq d(y) \leq y_N - \psi(y').
\end{equation*}

Then, by taking $\eta$ smaller if necessary, we can write
\begin{align}\label{upperI}
\begin{split}
I \leq & \int \limits_{\tilde Q \cap \Omega} \frac{(y_N - \psi(y'))^\tau  - \rho^\tau}{|x - y|^{N + 2s - \alpha}}dy + C \int \limits_{\tilde Q \cap \Omega^c} \frac{(y_N - \psi(y'))^\tau |y'|^2}{|x - y|^{N + 2s - \alpha}}dy \\
=: & I_{1} + I_{2}.
\end{split}
\end{align}

For $I_{2}$, using that $|x - y| \geq |y'|$ we have
\begin{align*}
I_{2} \leq & C \int_{B_\eta'} |y'|^{2 + \alpha - N - 2s} \int_{\psi(y')}^{\eta} (y_N - \psi(y'))^\tau dy_N dy' \\
\leq & C \frac{\eta^{\tau + 1}}{\tau + 1} \int_{B_\eta'} |y'|^{2 + \alpha - N - 2s} dy',
\end{align*}
concluding that 
$$
I_{2} = O(1)
$$
where the $O$-terms depends on $C_\Omega, N, s, \eta, \tau$, but not on $\rho$.

For $I_{1}$, performing the change $y = \rho^{-1} y$, and denoting $\tilde \psi(y') = \rho^{-1} \psi(\rho y')$ we have
\begin{equation*}
I_{1} = \rho^{\tau - 2s + \alpha} \int_{\rho^{-1} (\tilde Q \cap \Omega)} \frac{(y_N - \tilde \psi(y'))^\tau}{|e_N - y|^{N + 2s - \alpha}}dy.
\end{equation*}
\bigskip

Notice that $|e_N - y| \geq 1/(2\sqrt{N})$ for all $y \in \rho^{-1} (\tilde Q \cap \Omega)$. We consider the decomposition $\rho^{-1} (\tilde Q \cap \Omega) = A \cup B$, with 
\begin{equation*}
A = \{ y = (y', y_N) \in \rho^{-1} (\tilde Q \cap \Omega): \tilde \psi(y') < y_N < \tilde \psi(y') + 1/2 \}, \quad \mbox{and} \quad B = \rho^{-1} \tilde Q \setminus A.
\end{equation*}

Since for $y \in A$ we have $|e_N - y| \geq c(1 + |y'|)$ for some $c \in (0,1)$, we have
\begin{align*}
\int_{A}  \frac{(y_N - \tilde \psi(y'))^\tau}{|e_N - y|^{N + 2s - \alpha}}dy \leq & C \int_{B_\eta'} (1 + |y'|)^{-(N + 2s - \alpha)} \int_{\tilde \psi(y')}^{1 + \tilde \psi(y')} (y_N - \tilde \psi(y'))^\tau dy_N dy' \\
\leq & C \int_{B_\eta'} (1 + |y'|)^{-(N + 2s - \alpha)} dy', \\
\leq & C
\end{align*}
for some $C > 0$ not depending on $\rho$. On the other hand, since for $y \in B$ we have $(y_N - \tilde \psi(y')) \geq c$, and that $|e_N - y| \geq c(1 + |y|)$ for some $c \in (0,1)$, since $\alpha < 2s$ we have 
\begin{equation*}
\int_{B}  \frac{(y_N - \tilde \psi(y'))^\tau}{|e_N - y|^{N + 2s - \alpha}}dy \leq C \int_{B}  \frac{1}{|e_N - y|^{N + 2s - \alpha}}dy \leq C,
\end{equation*}
for some $C > 0$. This implies that
$$
I_1 = O(\rho^{\tau - 2s + \alpha})
$$

Joining the estimates for $I_1$ and $I_2$, and replacing into~\eqref{upperI}, together with~\eqref{lowerI}, we conclude that
\begin{equation*}
\B[\tilde Q](d^\tau, \xi)(x) = O(\rho^{\tau - 2s + \alpha}).
\end{equation*}

Using this estimate and~\eqref{BQint}, and replacing into~\eqref{BQeta}, we conclude the estimate for $\B[Q_\eta](d^\eta, \xi)(x) = O(\rho^{\tau + \alpha- 2s})$. We replace this and~\eqref{BQetac} into~\eqref{bilineal2} to arrive at
\begin{equation*}
\B(d^\tau, \xi)(x) = O(\rho^{\tau + \alpha - 2s}) + O(1) + O(\rho^\tau).
\end{equation*}

This, together with~\eqref{1} and~\eqref{2} lead us to the result.
%
%
%
%
%
%
%
%
\end{proof}

%
%
%
%
%

\section{Proof of Theorem~\ref{teo1}.}\label{sec4}

This section is entirely devoted to the 
\begin{proof}[Proof of Theorem~\ref{teo1}:] 
By compactness of $\partial \Omega$, we can consider an extension of the boundary condition $h$ (that we still denote as $h$) and a modulus of continuity $m$ such that
$$
|h(x) - h(y)| \leq m(|x - y|) \quad \mbox{for all} \ x,y \in \R^N.
$$

It is known that such a function $m$ satisfies 
\begin{equation}\label{modulus}
m(t) \leq m(\eta)(1 + t/\eta)
\end{equation}
for all $t, \eta > 0$, see~\cite{I}.


We start the construction of suitable barriers. 

First, using a power of the distance function with positive, small exponent $\beta \in (0,s)$ (see~\cite{DQT}), it is  possible to construct a solution $w_1$ to the problem
\begin{equation*}
\left \{ \begin{array}{rll} \M^+ w_1 & = -1 \quad & \mbox{in} \ \Omega, \\ w _1 & = 0 \quad & \mbox{in} \ \Omega^c, \end{array} \right .
\end{equation*}
such that $c d^\beta \leq w$ in $\Omega$, for some positive constant $c$.

Let $s - 1 < \tau < s - 1 + \alpha$. By the estimates in~\cite{DQTe}, a smooth interior modification of the power function $d^\tau$ allows us to construct a function $\tilde w_2$ such that $\M^+ \tilde w_2 \leq -d^{\tau -2s}$ in an open neighborhood of $\partial \Omega$ (relative to $\Omega$), such that $\M^+ \tilde w_2 \leq C_{\Omega'}$ for each $\Omega ' \subset \subset \Omega$, and equal to zero in the exterior of the domain. Using a linear combination of this function and $w_1$, we can construct a function $w_2$ with $w_2 \geq 0$ in $\Omega$, solving
\begin{equation*}
\left \{ \begin{array}{rll} \M^+ w_2 & \leq  -d^{\tau -2s} \quad & \mbox{in} \ \Omega, \\ w_2 & = 0 \quad & \mbox{in} \ \Omega^c, \end{array} \right .
\end{equation*}
in the viscosity sense.

Now, we adapt the strategy of~\cite{CDLP} to the unbounded setting to get the barriers with the precise boundary condition.
For each $y \in \partial \Omega$ and $\eta \in (0,1)$, we consider
$$
V_y^\eta(x) =\Big{(} h(y) + m(\eta) + \frac{m(\eta)}{\eta} |x - y| \Big{)} d_+^{s-1}(x) + C_2^\eta w_2(x), \quad x \in \R^N,
$$
with $C_2^\eta > 0$ to be chosen (of order $m(\eta)/\eta$). 

For each $y \in \partial \Omega$, we have $V_y = 0$ in $\Omega^c$. By~\eqref{modulus} and the fact that $w_2 \geq 0$, we see that $V_y(x)d^{1 - s}(x) \geq h(x)$ for all $x \in \Omega$. Then, using the estimates of Lemma~\ref{lematec} and the definition of $w$, for each $x \in \Omega$ we have
\begin{align*}
& \I V^\eta_y(x) \\
\leq & C \Gamma \| h \|_\infty d(x)^{-1} + C \Gamma \frac{m(\eta)}{\eta} (|x - y| d(x)^{-1} + d(x)^{-s}) + C_2^\eta \M^+ w_2(x) \\
\leq & C \Gamma \| h \|_\infty d(x)^{-1} + C \Gamma \frac{m(\eta)}{\eta} d(x)^{-1} - C_2^\eta d(x)^{\tau - 2s},
\end{align*}
and by the choice of $\tau$ we have $\tau - 2s < -1$, and taking $C_2^\eta = C_2 m(\eta)/\eta$ for some $C_2 > 0$ large enough, we conclude that
\begin{equation*}
\I V_y^\eta \leq -\frac{C_2^\eta}{2} d^{\tau - 2s} \quad \mbox{in} \ \Omega,
\end{equation*}
from which we have constructed a (strict) superharmonic function, taking the homogeneous exterior data.

 By standard results in the viscosity theory, we have the function
$$
V(x) = \inf \{ V_y^\eta(x) : \eta \in (0,1), \ y \in \partial \Omega \}, 
$$
is in $L^1_\omega(\R^N)$, it is strict superharmonic in $\Omega$, and satisfies the exterior condition $V = 0$ in $(\bar \Omega)^c$. It is direct to see that $V \in C(\Omega)$. 

On the other hand, taking $C_2^\eta$ as above, for each $y \in \partial \Omega$ and $\eta \in (0,1)$ we consider the function
$$
U_y^\eta(x) = \Big{(} h(y) - m(\eta) - \frac{m(\eta)}{\eta} |x - y| \Big{)} d^{s-1}_+(x) - C_2^\eta w_2(x), \quad x \in \R^N,
$$
is a viscosity subsolution to the problem, from which
$$
U(x) = \sup \{ U_y^\eta(x) : \eta \in (0,1), \ y \in \partial \Omega \}, 
$$
is continuous, subharmonic in $\Omega$ and $U = 0$ in $(\bar \Omega)^c$. Notice that for each $x \in \Omega$, and all $\eta \in (0,1), y \in \partial \Omega$ we have
$$
U(x) \leq h(x) d^{s-1}(x) \leq (h(y) + m(\eta) + m(\eta)|x - y|/\eta)d^{s-1}(x) + C_2^\eta w_2(x),
$$ 
from which $U(x) \leq V(x)$ in $\Omega$. Thus, invoking Proposition~\ref{Perron}, there exists a harmonic function $u \in C(\Omega)$ such that $U \leq u \leq V$, and equal to  zero in $(\bar \Omega)^c$.

Now we investigate the boundary behavior of the solution constructed in this way.

Recall $0 < \tau < s$. Notice that for all $x_0 \in \partial \Omega$, taking $y = x_0$ in the infimum defining $V$, and for all $\eta > 0$ we have
\begin{align*}
& d^{1 - s}(x) u(x) - h(x_0) \\
 \leq &  d^{1 - s}(x) V(x) - h(x_0) \\
\leq &  m(\eta) + \frac{m(\eta)}{\eta} |x - x_0| +
 C_2^\eta d^{s - 1}(x) w_2(x) \\
 \leq &  m(\eta) + \frac{m(\eta)}{\eta} |x - x_0| +
 CC_2^\eta d^{1 - s}(x) d^{\tau}(x) \\
  \leq & m(\eta) + (1 + C_2)\frac{m(\eta)}{\eta} (|x - x_0| + d^{1^-}(x)) \\
  \leq & m(\eta) + 2(1 + C_2)\frac{m(\eta)}{\eta} |x - x_0|^{1^-},
\end{align*}
where by $1^-$ we mean every exponent $\alpha \in (0,1)$. From here, taking infimum in $\eta \in (0,1)$ we conclude that $0 \leq u(x) - h(x_0) \leq \tilde m(|x - x_0|)$ for some modulus of continuity $\tilde m$ just depending on the data.



Similarly, for each $\eta > 0$ we have the estimate
\begin{align*}
& d^{1 - s}(x) u(x) - h(x_0) \\
 \geq & d^{1 - s}(x) U(x) - h(x_0) \\
\geq & - m(\eta) - 2(1 + C_2)\frac{m(\eta)}{\eta} |x - x_0|^{1^-} 
\end{align*}
and we conclude, by taking infimum in $\eta$, that $d^{1 - s}(x) u(x) - h(x_0) \geq \tilde m(|x - x_0|)$ for some modulus $\tilde m$. This concludes that $\lim_{x \to x_0} d^{1-s}(x) u(x) = h(x_0)$. The proof is now complete.
%
%
%
\end{proof}

As we mentioned in the Introduction, we can get the following
\begin{cor}\label{corfgh}
Assume the hypotheses of Theorem~\ref{teo1} hold. Assume that $f \in C_b(\Omega)$ and $g \in C(\bar \Omega^c) \cap L^1_\omega(\bar \Omega^c)$. Then, there exists a unique viscosity solution $u \in C^{1, \alpha}(\Omega)$ for the Dirichlet problem~\eqref{eqdefsol}.
\end{cor}
\begin{proof}
By the results of~\cite{BChI}, we can consider $u_1, u_2 \in C(\R^N)$ solving, respectively, the Dirichlet problems
\begin{align*}
\left\{\begin{array}{rll}
\mathcal M^- u_1& = \| f \|_\infty + 1& \mbox{in} \  \Omega,\\
u_1 & = g &  \mbox{in} \  \Omega^c, 
\end{array}\right.
\quad \mbox{and} \quad 
\left\{\begin{array}{rll}
\mathcal M^+ u_2 & = -\| f \|_\infty - 1& \mbox{in} \  \Omega,\\
u_2 & = g &  \mbox{in} \  \Omega^c.
\end{array}\right.
\end{align*}

Using that the $\I$ is uniformly elliptic in the sense of~\cite{CS}, that is
$$
\M^- (u - v) \leq \I(u) - \I(v) \leq \M^+(u - v), \quad \mbox{for suitable} \ u, v,
$$
and denoting by $w$ the fractional harmonic function in Theorem~\ref{teo1}, we conclude that $w + u_1$ is a (strict) subsolution to~\eqref{eqdefsol}, and $w + u_2$ is a viscosity supersolution for the same problem. In fact, they satisfy the exterior condition in $\bar \Omega^c$, and since $\mathcal M^- u_1 \leq \mathcal M^- u_2$ in $\Omega$, by comparison principle we have $w + u_1 \leq w + u_2$. Since $u_1, u_2$ are continuous up to the boundary, the limit boundary condition given by $h$ is not modified. By Proposition~\eqref{propcomparison} we have this solution is unique.
\end{proof}

\begin{remark}\label{rmkh=0}
Under the assumptions of Corollary~\ref{corfgh}, if we assume $h \equiv 0$ on $\partial \Omega$, then we have that the unique solution to~\eqref{eqdefsol} is the unique bounded solution to $\I u = f$ in $\Omega$ and $u = g$ in $\Omega^c$ obtained in~\cite{BChI}. This solution is continuous up to the boundary. 

Moreover, as a consequence of the barriers constructed in the proof of Theorem~\ref{teo1}, we have the family of solutions $\{ u_s\}_{s \in (0,1)}$   found in Corollary~\ref{corfgh} is uniformly bounded and equicontinuous in each compact subset of $\Omega$. Here we have stressed the dependence on the parameter $s$ in~\eqref{Kestable}. Since $d^{s - 1}$ tends to $1$ locally uniformly in $\Omega$ as $s \to 1^-$, by stability of viscosity solutions we have the limit function $u(x) = \lim_{s\to 1} u_s(x)$ is the unique solution to the second-order Dirichlet problem $F(Du) = f$ in $\Omega$, $u = g$ on $\partial \Omega$. For $M$ an $N \times N$ real matrix, $F(M)$ is given by
\begin{equation*}
F(M) = \inf_{i \in I} \sup_{j \in J} A_{ij}^k M_{kk}, \quad \mbox{with} \quad A_{ij}^k = \lim_{s \to 1^-} \frac{1}{2 - 2s} \int_{S^{N - 1}} q_k^2 a_{ij}(q) dS(q),   
\end{equation*}
where we have used the convention of summation over repeated indices. The entires $A_{ij}$ are well-defined and positive if the ellipticity constants $\gamma, \Gamma$ of the family are comparable to the normalizing constant $C_{N,s}$ in~\eqref{fracLapl}.
\end{remark}

\section{Proof of Theorem~\ref{teo2}.}\label{sec6}

In this section we provide the
\begin{proof}[Proof of Theorem \ref{teo2}:]
We start with the upper bound.
 For this, we fix $x \in \Omega$ close to the boundary, and denote $r = d(x)/4$. Then, the function $\tilde u(y) := u(x + ry)$ solves the equation
$$
-\I \tilde u = 0 \quad \mbox{in} \ B_1.
$$

Thus, employing the interior $C^{1, \alpha}$ estimates of~\cite{K}, Theorem 4.1, we see that
\begin{equation*}
[\tilde u]_{C^{1, \alpha}(B_{1/2})} \leq C (\| \tilde u \|_{L^1(\R^N, \omega)} + \| \tilde u \|_{L^\infty(B_1)}).
\end{equation*}

Using the estimates for the solution $u$, it is easy to see that 
$$
\| \tilde u \|_{L^\infty(B_1)} \leq C r^{s - 1},
$$
for some $C > 0$ not depending on $r$. On the other hand, using again the estimates for $u$ we see that
\begin{align*}
\| \tilde u \|_{L^1_\omega(\R^N)} \leq & C r^{s - 1} \int_{B_{1/r}} \omega(y)dy + \int_{B_{1/r}^c} |u(x + ry)|\omega(y)dy \\
\leq & C r^{s - 1} + r^{2s} \int_{B_1^c} |u(x + z)||z|^{-(N + 2s)}dy,
\end{align*}
and using that $u \in L^1(\Omega) \cap L_\omega^1(\R^N)$, we conclude the existence of a constant $C > 0$ such that 
\begin{align*}
\| \tilde u \|_{L^1_\omega(\R^N)} \leq C r^{s - 1}.
\end{align*}

Then, we conclude that $[\tilde u]_{C^{1, \alpha}(B_{1/2})} \leq Cr^{s - 1}$. Scaling back and using that $r$ is proportional to $d(x)$, we conclude that $|Du(x)| \leq C d(x)^{s - 2}$. 

\medskip

For the lower bound for the gradient, we argue by contradiction, assuming the existence of a sequence  $z_j \in \Omega$ with $d_j := d(z_j) \to 0$ as $j \to \infty$, and such that
\begin{equation}\label{absurdo}
d_j^{s-2}|Du(z_j)| \to 0 \quad \mbox{as} \ j \to \infty.
\end{equation}

Let $\hat z_j \in \partial \Omega$ the (unique) projection of $z_j$ onto $\partial \Omega$. Without loss of generality, we can assume $\hat z_j \to 0$ and that $Dd(0) = e_N$. In particular, by the regularity of the boundary, we have $Dd(\hat z_j) = e_N + o_j(1)$ for some $o_j(1) \to 0$ as $j \to \infty$.

Now, consider the function 
$$
v_j(y)=d_j^{1-s}u(\hat z_j + d_j y).
$$

Denote $\Omega_j = d_j^{-1}(\Omega - \hat z_j)$. For each $y \in \R^N_+$, there exists $j_0 = j_0(y)$ large enough such that $y \in \Omega_j$ for all $j \geq j_0$. Notice that for $y \in \Omega_j$ we have 
$$
d(\hat z_j + d_j y) = d_j d_{\partial \Omega_j} (y),
$$
and from here, we notice the continuity estimates given in the proof of Theorem~\ref{teo1} allow us to write
\begin{align*}
d_{\partial \Omega_j}^{1 - s}(y) v_j(y) \leq & h(\hat z_j + d_j y) + \tilde m(|\hat z_j + d_j y|) \\
\leq & h(0) + o_j(1) + \tilde m(|\hat z_j + d_j y|)
\end{align*}
for all $y \in \R^N_+$ and all $j$ large enough. A similar lower bound can be established, and we can summarize them as
\begin{equation*}
h(0) - o_j(y) \leq d_{\partial \Omega_j}^{1 - s}(y) v_j(y) \leq h(0) + o_j(y), \quad y \in \Omega_j,
\end{equation*}
where $o_j(y) \to 0$ as $j \to \infty$, locally for $y \in \R^N_+$.
From here we conclude that $\{ v_j \}_j$ is equibounded in each compact set $K \subset \subset \R^N_+$. Thus, we see that $v_j \to V$ as $j \to \infty$, locally in $\R^N \setminus \R^{N - 1} \times \{0\}$, and taking limit as $j \to \infty$ we conclude that
$$
V(y) = h(0) (y_N)_+^{s-1}, \quad y \in \R^N_+.
$$

Using the  $C^{1, \alpha}$ estimates for the solution $u$, we can assume (up to subsequences) that $\{ v_j\}$ converges in $C^{1, \alpha}_{loc}(\R^N_+)$. Since $y_j := Dd(\hat z_j) \in B_{1/2}(e_N)$ for all $j$ large enough, we have that
$$
d_j^{2 - s} Du(\hat z_j + d_j y_j) = Dv_j(y_j) \to DV(e_N) = h(0) (s-1) e_N,
$$
as $j \to \infty$, but this is a contradiction with~\eqref{absurdo}. 
\end{proof}

\begin{remark}\label{rmkfinal}
We expect that if $h$ touches zero, then the lower bound for the gradient may fail to hold. We conjecture this through the following explicit computation in the half-space:
let
$$
u(x', x_N) = p' \cdot x' x_N^{s-1},
$$
for a nonzero vector $p' \in \R^{N - 1}$.

This function is $\Delta^s$-harmonic in the upper half-space. In fact, denoting $u_1(x) = (x_{N})_+^{s - 1}$ and $u_2(x) =  p' \cdot x'$ and noting that both functions are $s$-harmonic, for each $x = (x',x_N) \in \R^N_+$ we can write
\begin{align*}
& -(-\Delta)^s u(x) \\
= & 2 \mathcal B(\phi_1, \phi_2)(x) \\
= & C_{N,s} \int_{\R^N} \frac{(u_1(y) - u_1(x))(u_2(y) - u_2(x))}{|x - y|^{N + 2s}}dy \\
= & C_{N,s} \int_{\R} ((y_N)_+^{s - 1} - x_N^{s - 1}) \int_{\R^{N-1}} \frac{p' \cdot (y' - x')}{|x - y|^{N + 2s}}dy' dy_N \\
= & C_{N,s} \int_{\R} \frac{(y_N)_+^{s - 1} - x_N^{s - 1}}{|y_N - x_N|^{1 + 2s}} \int_{\R^{N-1}} \frac{p' \cdot (y' - \frac{x'}{|y_N - x_N|})}{(1 + |\frac{x'}{|y_N - x_N|} - y'|^2)^{\frac{N + 2s}{2}}}dy' dy_N \\
= & C_{N,s} \int_{\R} \frac{(y_N)_+^{s - 1} - x_N^{s - 1}}{|y_N - x_N|^{1 + 2s}} \int_{\R^{N-1}} \frac{p' \cdot z'}{(1 + |z'|^2)^{\frac{N + 2s}{2}}}dz' dy_N,
\end{align*}
and by the symmetry of the kernel we conclude the inner integral term vanishes.
Notice that
\begin{equation*}
Du(0', x_N) = x_N^{s-1} (p', 0),
\end{equation*}
which prevents the lower bound in Theorem~\ref{teo2} to hold when the boundary data vanishes on $\partial \Omega$. We expect to have similar behavior for general, bounded smooth domains.

Finally, we shall mention that we do not know if $u$ above is is the unique solution to problem~\eqref{eqdefsol} with $\Omega = \R^N_+$, $f, g \equiv  0$ and $h(x') = p' \cdot x'$.
\end{remark}

\medskip

\noindent {\bf Acknowledgements:}  G. D. was partially supported by Fondecyt Grant 1190209.
A. Q. was partially supported by Fondecyt Grant No. 1190282 and Programa Basal, CMM. U. de Chile.
E.T. was partially supported by Fondecyt Grant No. 1201897.


%
%
%

%

\end{document}